\documentclass[12pt]{article}
\usepackage{hyperref}
\usepackage{mathrsfs,bm,amssymb,color}
\usepackage{latexsym,amscd}
\usepackage[sf,bf,tiny,center]{titlesec}
\usepackage{fancyhdr}
\usepackage{float}
\usepackage{hyperref}
\usepackage{color}
\usepackage{booktabs}
\usepackage{multirow}
\usepackage{graphicx}
\usepackage{url}
\usepackage{color,hyperref}

\usepackage{graphicx,amsmath,amsfonts,latexsym,amssymb,amsthm,amscd}
\usepackage{latexsym,amscd}
\newcommand{\eq}[1]{\begin{equation}\label{#1}}
\newcommand{\en}{\end{equation}}
\newcommand{\eqst}[1]{\begin{equation*}\label{#1}}
\newcommand{\enst}{\end{equation*}}
\newcommand{\eqr}[1]{\begin{eqnarray}\label{#1}}
\newcommand{\enr}{\end{eqnarray}}
\newcommand{\eqrst}[1]{\begin{eqnarray*}\label{#1}}
\newcommand{\enrst}{\end{eqnarray*}}

\makeatletter
\@addtoreset{equation}{chapter}
\makeatother

\numberwithin{equation}{section}

\newcommand{\ep}{\end{proposition}}
\newcommand{\bc}[1]{\begin{corollary}\label{#1}}
\newcommand{\ec}{\end{corollary}}
\newcommand{\bdf}[1]{\begin{definition}\label{\rm #1}}
\newcommand{\edf}{\end{definition}}
\newcommand{\bt}[1]{\begin{theorem}\label{#1}}
\newcommand{\et}{\end{theorem}}
\newcommand{\bl}[1]{\begin{lemma}\label{#1}}
\newcommand{\el}{\end{lemma}}
\newcommand{\bp}[1]{\begin{proposition}\label{#1}}
\newcommand{\br}[1]{\begin{remark}\label{#1}}
\newcommand{\er}{\end{remark}}
\newcommand{\bi}{\begin{description}}
\newcommand{\ei}{\end{description} }

  \newcommand{\beq}{\begin{equation}}
  \newcommand{\eeq}{\end{equation}}

  \newtheorem{theorem}{Theorem}[section]
\newtheorem{definition}[theorem]{Definition}
\newtheorem{remark}[theorem]{Remark}
  \newtheorem{lemma}[theorem]{Lemma}

\newtheorem{proposition}[theorem]{Proposition}

\newtheorem{corollary}[theorem]{Corollary}

\begin{document}

\begin{center}
\large{\textbf{Perturbation and bifurcation analysis of a gonorrhoea 
dynamics model with control}}
\end{center}

\begin{center}
Louis Omenyi$^{1,3,\ast , 
\href{https://orcid.org/0000-0002-8628-0298}{id}},$
 Aloysius Ezaka$^{2},$ Henry O. Adagba$^{2},$  
 Friday Oyakhire$^{1},$  Kafayat Elebute$^{1},$ 
Akachukwu Offia$^{1}$ and Monday Ekhator$^{1}$ \\
$^{1}$ Department of Mathematics and Statistics,  \\
Alex Ekwueme Federal University, Ndufu-Alike,  Nigeria\\
$^{2}$Department of Industrial Mathematics and Applied Statistics, \\
Ebonyi State University, Abakaliki, Nigeria\\
$^{3}$Department  of  Mathematical Sciences,  \\
Loughborough University, Leicestershire, United Kingdom\\ 
 Corresponding author, email: \url{omenyi.louis@funai.edu.ng} 
\end{center}

\begin{abstract}
A model for the transmission dynamics of gonorrhoea with control 
incorporating passive immunity is formulated. We show that introduction 
of  treatment or control parameters leads to transcritical bifurcation. 
The  backward bifurcation coefficients were calculated and their 
numerical perturbation results to different forms of equilibria. 
The calculated effective reproduction number  of the model with control 
is sufficiently small. This implies  asymptotically stability of the 
solution, thus, the disease can be controlled in a limited time.     
\end{abstract}
 
\section{Introduction}
Due to increasing rate of infertility among the teaming population as a 
result  of sexually transmitted infections, it becomes necessary to 
undertake  prompt prevention and control activities to tackle the ugly 
incidence of sexually transmitted diseases \cite{Gre}.  
Gonorrhoea is one of such sexually transmitted infectious diseases 
caused by a bacterium called Neisseria gonorrhoeae \cite{Une}. 
The neisseria gonorrhoea is characterized by a very short period of 
latency, namely, $2 -10$ days  \cite{Mus2} and is commonly found in the 
glummer epithelium such as the urethra and endo-cervix epithelia of the 
reproductive track \cite{Gar}. Gonorrhoea is transmitted to a new born 
infant from the infected mother through the birth canal thereby causing 
inflammations and eye infection such as conjunctivitis. It 
is also spread through unprotected sexual intercourse, \cite{Ugw}.

Studies by Usman and Adam \cite{Usm} and Center for Disease Control 
Report in show that male patients of gonorrhoea have  
pains in the testicles (known as epididymitis), painful urination due 
to scaring inside the urethra while in female patients, the disease 
may  ascend the genital tract and block the fallopean tube  leading to 
pelvic inflammatary disease (PID)  and infertility, see also 
\cite{Ram}. Other complications associated with this epidemic include 
arthritis, endocarditis, chronic pelvic pain, meningitis and ectopic 
pregnancy, \cite{Ril}. 

Gonorrhoea confers temporal immunity on some individuals in the 
susceptible class while some others are not immuned, \cite{Ugw}. This 
immunity through the immune system plays an important role in 
protecting the body against the infection and other foreign substances, 
\cite{CDC}. That is why an immuno-compromised 
patient has a reduced ability to fight infectious disease such as 
gonorrhoea due to certain diseases and genetic disorder, \cite{Sch}. 
Such patient may be particularly vulnerable to opportunistic infection 
such as gonorrhoea. Hence, immune reaction can be stimulated by 
drug-induced immune system such as Thrombocytopenia, \cite{Sch}. 
This helps to reduce the waning rate of passive immunity in the immune 
class, \cite{Bas}. However, if the activity of 
immune system is excessive or over-reactive due to lack of cell 
mediated immunity, a hypersensitive reaction develops such as auto-
immunity and allergy which may be injurious to body or may even cause 
death \cite{WHO}.

Statistically, gonorrhoea infection has spread worldwide with more 
than $360$ million new cases witnessed globally in adults aged 
$15-49$ years, \cite{CDC}.  In 1999, above over $120$ million 
people in African countries were reported to have contracted the 
disease. While over $82$ million people were reported in Nigeria, 
\cite{CDC}. Researches abound on the modelling and control of this 
epidemic with various approaches and controls, see e.g. 
\cite{CDC, Jin, Mus1, Sac, Ugw, Une} and mostly 
recently \cite{Ibr, Whi, Osn} and \cite{Did}. This present study 
continues the discussion by incorporating passive immunity in the model 
and introducing control measures capable of eliminating the 
disease in Nigeria. To validate the claim, we employ perturbation 
and bifurcation of the model variables and parameters and 
mathematically analyse the stability of the system. This underscores  
the role of mathematical analysis of models to elicit desired 
results, see e.g. \cite{Om19} and \cite{Om21}. Education and 
enlightenment, use of condom and treatment of patients with ampilicin 
and azithromycin are the control measures adopted to eradicate the 
disease.
  
\section{Materials and Methods}
To formulate the model, in time $t,$ we let $Q(t)$ be passive immune 
class, $S(t)$ the susceptible compartment, $L(t)$ the  latent  class  
$I(t),$ the infectious class, $T(t),$ the  treated class and  $R(t)$ 
be the recovered compartment. Let the parameters of the model 
$\sigma$  as level of  recruitment, $\upsilon$ as waning rate of 
immunity, $\mu$  as rate of natural mortality, $\lambda$ as contact 
rate  between the susceptible and the latent classes, $\eta$ as  
treatment rate of latent class, $\gamma$ as induced death rate due to 
the infection, $\alpha$ as treatment rate of infected compartment, 
$\beta$ as  infectious rate of Latent class, $\omega$ as recovery rate 
of treated class, $\delta$ as rate at which recovered class become
susceptible again, $\theta$ as infectious rate from the susceptible 
class direct to the infectious class, $k_{1}$ as control measure given 
to latent class as $k_{2}$ as control measure given to infected class. 

We assume that recruitment into the population is by birth or 
immigration; all the parameters of the model are positive, some 
proportions of new birth are immunized against the 
infection; the immunity conferred on the new birth wanes after 
sometime, and that the rate of contact of the disease due to 
interaction $\lambda$ rate is due to the movement of the infected 
population. Consequently, the total population at time $t$  is 
\[ N(t) = Q(t) + S(t) + L(t) + I(t) + T(t) + R(t). \]

So, the flow diagram of the model is shown as figure \eqref{model2}.
\begin{figure}[!ht]
\begin{center}
\includegraphics[width=0.7\textwidth]{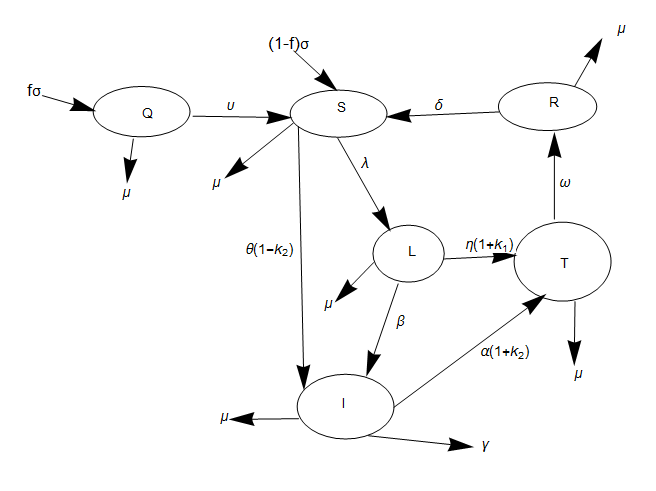}
\caption{The Extended Model with Control.}
\label{model2} 
\end{center}
\end{figure}
So, the model for the gonorrhoea transmission dynamics is given by the 
following deterministic systems of non-linear differential equations  
\eqref{3.9b}: 
\begin{eqnarray} \label{3.9b}
\left. \begin{array}{rcl}
\frac{dQ}{dt} & =& f \sigma - \upsilon  Q -\mu Q \\
\frac{dS}{dt} & =& \upsilon  Q + (1 - f) \sigma - \theta S(1 - k_{2} ) 
+  \delta R - \mu S - \theta S I \\
\frac{dL}{dt} & =& \theta S I - \beta L - \mu L - \eta(1 + k_{1} )L \\
\frac{dI}{dt} & =& \beta L  + \theta S (1-k_{2}) - ((\mu + \gamma)  +  
\alpha(1 +  k_{2} ))I \\
\frac{dR}{dt} & =& \omega T - \mu R - \delta R \\
\frac{dT}{dt} & =& \eta(1 + k_{1} )L + \alpha(1 + k_{2} )I - \mu T -
\omega T .
\end{array} \right\}
\end{eqnarray}
 
 We will use that bifurcation theory states that perturbation in the 
 parameter of a model leads to a change in the behaviour of the 
 equilibrium solution, \cite{Gar}. In the model, we use the center 
 manifold method to assess  the direction of bifurcation (i.e, either 
 forward or backward). The method reduces the system to a smaller 
 system which has the same qualitative properties and  can be studied 
 in a relatively easier way, \cite{Bas}.  This leads to a result on 
 endemic equilibrium and backward bifurcation for our model. 
 
Besides, the theory of epidemiology signifies the phenomenon of 
backward bifurcation,  that is  the classical requirement 
the model's effective reproduction number $R_{e} < 1.$  Although this 
is necessary, it is no longer sufficient to conclude the effective 
control or elimination of gonorrhoea in a population, see e.g. 
 \cite{WHO}. Therefore, in this model we consider the 
nature of the equilibrium solution near the bifurcation point 
$R_{e} = 1$ in the neighbourhood of the disease-free equilibrium 
$(E_{0} ).$ The disease-free equilibrium is locally 
asymptotically stable if $R_{e} < 1$ and unstable if $R_{e} > 1.$ 
But when $R_{e} = 1,$ another equilibrium point bifurcates from the 
disease-free equilibrium.  In this case, the disease would invade the 
population in the case of backward bifurcation, \cite{Gre}.

\section{Results}
We first observe that setting the right hand side of the system 
\eqref{3.9b} to zero gives the disease-Free Equilibrium (DFE) of the 
model as the equilibria: 
\begin{equation*}
(Q^0 ,S^0 ,L^0 , I^0 ,R^0 ,T^0 ) = \frac{f \sigma}{\mu + \upsilon}  ,   
\frac{\upsilon f \sigma + (\mu + \upsilon)(1 - f) \sigma}{(\mu + 
\upsilon)(\theta + \mu )} .
\end{equation*}
Now suppose 
\[L \neq 0, I \neq 0, R \neq 0 ~~ \text{and} ~~ T \neq 0\]
then the model attains endemic equilibrium and solving the endemic 
equilibria system of the model gives the endemic state to be 
\begin{eqnarray*}
Q^* & =& \frac{f \sigma}{\mu + \upsilon};\\
S^* & =& \frac{(\mu + \delta)(\mu + \omega) f \sigma 
+ (\mu + \upsilon)(\mu + \delta)(\mu + \omega) \sigma(1 - f) 
+ (\mu + \upsilon) \delta \omega(\alpha + \eta)}{(\mu 
+ \upsilon)(\mu + \delta)(\mu + \omega )} ; \\
L^* & =& \frac{(\lambda)(\mu + \delta)(\mu + \omega) f \sigma 
+ (\mu + \upsilon)(\mu + \delta)(\mu + \omega) \sigma(1 - f)
 + (\mu + \upsilon) \delta \omega(\alpha + \eta) }{(\mu + \beta 
 + \eta)(\mu + \upsilon)(\mu + \delta)(\mu + \omega ) }; \\
I^* & =& \frac{(\mu + \delta)(\mu + \omega) f \sigma 
+ (\mu + \upsilon)(\mu + \delta)(\mu + \omega) \sigma(1 - f) 
+ (\mu + \upsilon) \delta \omega(\alpha + \eta)(\beta \lambda 
+ ( \mu + \beta + \eta) \theta )}{(\mu + \alpha + \gamma)(\mu 
+ \beta + \eta)(\mu + \upsilon)(\mu + \delta)(\mu + \omega )}; \\
R^* & =& \frac{\omega(\alpha + \eta)}{(\mu + \delta)(\mu + \omega)} ;\\
T^* &=& \frac{\alpha + \eta}{\mu + \omega}.
\end{eqnarray*}

\begin{lemma}
A qualitative change in the behaviour of the equilibria 
due to perturbation results in bifurcation.  
\end{lemma}

\begin{proof}
For $\mu_{0} ,\mu_{1} > 0,$ it follows that the model is 
stable and that at steady state:  
\[\frac{d Q}{d t}= 0, ~~ \frac{d S}{d t}= 0, ~~ 
\frac{d L}{d t}= 0, ~~ \frac{d I}{d t}= 0, ~~ 
\frac{d R}{d t}= 0 ~~  \text{and} ~~ \frac{d T}{d t}= 0.\]
Thus, 
\begin{eqnarray}
 \frac{dQ}{dt} & =& f \sigma - \mu_{2} Q \nonumber \\
 \frac{dS}{dt} & =& (1 - f) \sigma + \upsilon Q + \delta R 
 - \theta S(1 -k_{2} )  - \mu S - \theta I S \nonumber \\
 \frac{dL}{dt} & = & \theta I S  - \mu_{1} L \label{4.4.16}\\
 \frac{dI}{dt} & =& \beta L + \theta S(1 - k_{2} ) - \mu_{0} I \label{4.4.17}\\
 \frac{dR}{dt} & =& \omega T - \mu_{2} R \nonumber\\
 \frac{dT}{dt} & =& \eta(1 + k_{1} )L +
  \alpha(1 + k_{2} )I - \mu_{3} T  .\nonumber
\end{eqnarray}
So letting 
\begin{eqnarray*}
\mu_{0} & = & \mu + \alpha + \gamma \\
\mu_{1} & = & \mu + \beta + \eta \\
\mu_{2} & = & \mu + \upsilon \\
\mu_{3} & = & \mu + \delta . 
\end{eqnarray*}
At steady state, the equilibrium points of  \eqref{4.4.16}  become
\[ 0 = \theta I S  - \mu_{1} L \Rightarrow 
L = \frac{\theta I S}{\mu_{1}}\Rightarrow 
L = ( 0, \frac{\theta I S}{\mu_{1}}). \]

While  the equilibrium points of equation \eqref{4.4.17} become
\[ 0 = \beta L + \theta S(1 - k_{2} ) - \mu_{0} I \Rightarrow 
I = \frac{\beta L + \theta S(1 - k_{2} )}{\mu_{0}} \Rightarrow 
I = ( 0, \frac{\beta L + \theta S(1 - k_{2} )}{\mu_{0}} ).\]
\end{proof}
This result is consistent with those of perturbed systems in 
\cite{Jin} and \cite{Gre}.

We have the next result.
\begin{proposition}
The disease dynamics is controllable in the population with a sufficient perturbation for sufficiently long time.
\end{proposition}

\begin{proof}
As shown above, the introduction of treatment (or control) parameter 
changes the initial stage of the infection, hence, transcritical 
bifurcation. Now adding  small perturbations to the equilibrium points 
of the model subject to changes in control or bifurcation parameter, we 
have
\[L = 0 + \varepsilon L =  
\frac{\theta I S }{\mu_{1}} + \varepsilon L\]
and 
\[ I = 0 + \varepsilon I  
=  \frac{\beta L + \theta S(1 - k_{2} )}{\mu_{0}}  + \varepsilon I .\]
Similarly,
\[ \frac{dL}{dt} = \theta I S  - \mu_{1} L 
 =   \theta I S  - \mu_{1} (\frac{\theta I S}{\mu_{1}}
  + \varepsilon L)
 = -\mu_{1} \varepsilon L.\]
Solving this gives 
 \begin{equation}\label{star}
L(t) = B e^{-\mu_{1} \varepsilon t}
\end{equation}
where $B$ is an arbitrary constant. Clearly,  
$|L| \rightarrow 0$ as  $|t| \rightarrow \infty.$

Observe that equation \eqref{star}  indicates that there is stability 
for all $\mu_{1} > 0.$ This means that the infection can be controlled 
in the population.

Moreover
\[\frac{dI}{dt} = \beta L + \theta S(1 - k_{2} ) - \mu_{0} I
= \beta L + \theta S(1 - k_{2} ) - \mu_{0} \varepsilon I.\]
Solving this gives 
\begin{equation} \label{4.4.27}
I(t) = Ae^{-\mu_{0} \varepsilon t} 
\end{equation} 
for an arbitrary constant $A.$ So, there is linear 
stability for all $\mu_{0} > 0.$  Moreover, 
\[|I| \rightarrow 0 ~~\text{as} ~~ |t| \rightarrow \infty .\]
On the addition of treatment or control parameters, we have the 
bifurcation shown in graphically in  Figure \ref{model1}.
\begin{figure}[!ht]
\begin{center}
\includegraphics[width=0.6\textwidth]{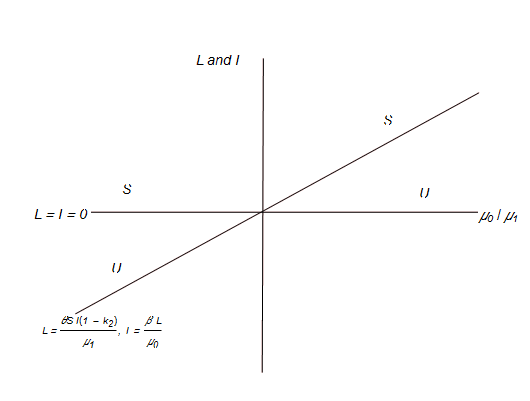}
\caption{The Transcritical Bifurcation of the gonorrhoea model with 
passive immunity.}
\label{model1} 
\end{center}
\end{figure}
\end{proof}

When one considers the basic reproduction number  $R_{0},$ which is the 
expected number of secondary infection produced in a completely 
susceptible population by a typical or one infected individual  
\cite{Van}, other results of this analysis follow.
The basic reproduction number is an important parameter used to 
determine how long an infectious disease can last or prevail in a given 
population. When $R_{0} < 1,$ it means that with time the  disease will 
die out of the population thereby giving it a clean health bill 
\cite{Gar}. But if $R_{0} > 1,$ it is expected that the disease will 
persist in the population. So for the disease to
die out of the population, the associated reproduction number  must be 
less than $1$ \cite{Het}. When control measure is given to a model, the 
reproduction number of the infectious disease becomes effective 
reproduction number $R_{e},$   \cite{Hoo}. 

\begin{proposition}
The controls in the model system  \eqref{3.9b} for the gonorrhoea 
dynamics extinct the pandemics from the population.
\end{proposition} 

\begin{proof}
For the infectious classes are L, I and  T, let 
\begin{equation*}
f_{i} = \begin{bmatrix}
\theta  I S \\
\theta(1 - k_{2} ) S \\
0
\end{bmatrix}
\end{equation*}
So that \begin{equation*}
\frac{\partial f_{i}}{\partial x_{j}} E_{0} = F =
 \begin{pmatrix}
0 & \theta S & 0 \\
0 & 0 & 0 \\
0 & 0 & 0
\end{pmatrix}.
\end{equation*}
Also
\begin{equation*}
v_{i} = \begin{bmatrix}
\beta L + \mu L + \eta(1 + k_{1} )L \\
\mu I + \gamma I + \alpha(1 + k_{2} )I - \beta L 
- \theta S(1 -k_{2} )\\
\mu T + \omega T - \eta(1 + k_{1} )L - \alpha(1 + k_{2} )I
\end{bmatrix}.
\end{equation*}
So that 
 \begin{equation*}
\frac{\partial v_{i}}{\partial x_{j}} E_{0} = V =
 \begin{pmatrix}
(\beta + \mu + \eta(1 + k_{1} )) & 0 & 0 \\
- \beta & (\mu + \gamma + \alpha(1 + k_{2} )) & 0 \\
- \eta(1 + k_{1} ) & - \alpha(1 + k_{2} ) & (\mu + \omega)
\end{pmatrix}.
\end{equation*}

The matrix formed by the co-factors of the determinant is
\begin{equation*}
\small{
\begin{pmatrix}
(\mu + \gamma + \alpha(1 + k_{2} ))(\mu + \omega) & -\beta(\mu + 
\omega) &
\alpha \beta(1 + k_{2} ) + \eta(1 + k_{1} )(\mu + \gamma + \alpha(1 + 
k_{2} ) \\
0 & (\beta + \mu + \eta(1 + k_{1} ) )(\mu + \omega) & - \alpha(1 + 
k_{2} )(\beta
+ \mu + \eta(1 + k_{1} )) \\
0 & 0 & (\beta + \mu + \eta(1 + k_{1} )(\mu + 
\gamma + \alpha(1 + k_{2} ))
\end{pmatrix}}
\end{equation*}
so that 
\begin{equation*}
V^{-1} = 
\begin{pmatrix}
\frac{1}{\beta + \mu + \eta(1 + k_{1} )} & 0 & 0 \\
\frac{\beta}{(\beta + \mu + \eta(1 + k_{1} ))(\mu 
+ \alpha(1 + k_{2} ) +
\gamma)} & \frac{1}{\mu + \gamma + \alpha(1 + k_{2} )} & 0 \\
\frac{\alpha \beta(1 + k_{2} ) + \eta(1 + k_{1} )}{(\beta
 + \mu + \eta(1 + k_{1} )(\mu + \omega)}  & 
\frac{-\alpha(1 + k_{2} )}{\mu + \gamma + \alpha(1 + k_{2} )(\mu
 + \omega)} & \frac{1}{\mu + \omega}
\end{pmatrix}.
\end{equation*}

Also,
\begin{equation*}
|F V^{-1} - \lambda I| = \begin{vmatrix}
\frac{\beta \theta S }{(\beta + \mu + \eta(1 + k_{1} ))(
\mu + \gamma + \alpha(1 + k_{2} ) )}  
- \lambda &  \frac{\theta S(1 - k_{2} )}
{\mu + \gamma + \alpha(1 + k_{2} )} & 0 \\
0 & 0 - \lambda & 0 \\
0 & 0 & 0 - \lambda
\end{vmatrix} = 0.
\end{equation*}
Hence,
 \begin{equation*}
\lambda^{2} (\frac{(\beta \theta S )}{(\beta + \mu + \eta(1 + k_{1} ))
(\mu + \gamma + \alpha(1 + k_{2} ))} - \lambda ) = 0.
\end{equation*}

Thus, either 
\[\lambda^{2} = 0 ~~\text{or}~~
\lambda = \frac{(\beta \theta S )}{(\beta + \mu + 
\eta(1 + k_{1} ))(\mu + \gamma + \alpha(1 + k_{2} ))}.\]

Therefore, the effective reproduction number
\begin{equation}\label{effrepno}
 R_{e} = \frac{(\beta \theta S )}{(\beta + \mu 
 + \eta(1 + k_{1} ))(\mu + \gamma + \alpha(1 + k_{2} ))} .   
\end{equation}
\end{proof}

To illustrate this, let our variables and parameters be as in Table 
\ref{t2}: 
\begin{table}[!!ht]
\centering
  \resizebox{0.8\textwidth}{!}{\begin{minipage}{\textwidth}
 \centering
\begin{tabular}{|l|c|c|c|c|c|c|c|c|c|c|r|}
\hline 
Parameter/Variable  & $\beta$  &  $\theta$ & $\mu $ & $\eta $ & 
$\gamma $ &  $\alpha $ &  $\delta $ &  $\upsilon $ & $\omega $  & 
$\sigma $ \\
\hline
Value & $0.01$ & $0.5$ &  $0.2$ & $0.1$ &  $0.01$ &  $0.2$ & $0.8$ & 
$0.4$ &   $0.7$ &   $0.4$ \\
\hline
Parameter/Variable & $d_1= k_{1} $ &  $d_2=k_{2} $ &   
$f$ & $S$  &  $Q$ &  $R$  &  $T$ &   $L$ &  $I$ & \\
\hline
Value  & $0.5$ & $0.8$ &   $0.91$ & $2000$ & $1000$ & 
$500$ & $1000$ & $1000$ & $500$ &  \\
\hline 
\end{tabular}
\label{t2}
\end{minipage}}
\caption{Parameters/variables and values.}
\end{table}
then, 
\begin{equation}\label{4.5.4}
R_{e} = \frac{\sigma \beta \theta((\mu + \upsilon) - \mu f)}{\mu (\mu + 
\alpha + \gamma) (\mu + \beta + \eta)
(\mu + \upsilon)}  = 0.09700176367 < 1.
\end{equation}
We have the following main result.
\begin{theorem}\label{thm1} 
The gonorrhoea model undergoes backward bifurcation at $R_{e} = 1$ 
whenever the bifurcation co-efficient $a$ and $b$ are positive.
\end{theorem}

\begin{proof}
Now, recall the effective reproduction number $R_{e}$ of the 
gonorrhoea infection as shown by equation \eqref{re2}
 
\begin{equation}\label{re2}
R_{e} = \frac{S \beta \theta }{\mu + \gamma + \alpha( 1 + k_{2} )(\mu + 
\beta + \eta( 1 + k_{1} )}
\end{equation}
Or \begin{equation}
R_{e} = \frac{\sigma \beta \theta((\mu + \upsilon) - \mu f)}{\mu (\mu + 
\alpha + \gamma) (\mu + \beta + \eta)
(\mu + \upsilon)}  = 0.09700176367 < 1.
\end{equation}

Let $\psi = \theta s$  be the parameter by which the bifurcation occurs 
at $R_{e} = 1.$
  
 Equation \eqref{re2} becomes
 \begin{eqnarray*}
 1 & =& \frac{\psi \beta }{\mu + \gamma +\alpha(1 + k_{2} ) \mu + \beta + \eta(1 + k_{1} )} \\
 \psi & =& \frac{\mu + \gamma +\alpha(1 + k_{2} )}{\beta }; ~~ 
 \beta \neq 0 .
\end{eqnarray*}

Let $x_{1} = Q,$ $x_{2} = S,$ $x_{3} = L,$ $x_{4} = I,$ $x_{5} = R,$ 
and $x_{6} = T.$ Furthermore, by using the vector notation, 
\begin{equation*}
X = (x_{1} ,x_{2} ,x_{3} ,x_{4} ,x_{5} ,x_{6} )^T
\end{equation*}
The model can be written in the form 
\begin{equation*}
\frac{dx}{dt} = ( f_{1}, f_{2} ,f_{3} ,f_{4} ,f_{5} ,f_{6} )^T
\end{equation*}
then  the  model equations \eqref{3.9b} become
\begin{eqnarray*}
 f_{1} & =& f \sigma - (\mu + \upsilon) x_{1} \\
 f_{2} & =& (1 - f) \sigma + \upsilon x_{1} + \delta x_{5} 
 - \psi(1 -k_{2} ) - \psi x_{4} - \mu x_{2} \\
 f_{3} & = & \psi x_{4} - \mu  x_{3} - \beta x_{3} 
 - \eta(1 + k_{1} ) x_{3} \\
 f_{4} & =& \beta x_{3} + \psi(1 - k_{2} ) - \mu x_{4} - \alpha(1 + k_{2} )x_{4} - \gamma x_{4}  \\
 f_{5} & =& \eta(1 + k_{1} )x_{3} +
  \alpha(1 + k_{2} )x_{4} - \mu x_{5} \omega x_{5} \\ 
 f_{6} & =& \omega x_{5} - \mu x_{6} - \delta x_{5} .
\end{eqnarray*}
Here $\mu_{2} = \mu + \upsilon,$  $\mu_{3} = \mu + \delta$ and $\mu_{4} 
= \mu + \omega .$
 
The Jacobian matrix at DFE is  therefore given by 
\[ J = \begin{pmatrix}
-(\mu + \upsilon) & 0 & 0 & 0 & 0 & 0 \\
\upsilon & -\mu & 0 & -\psi & 0 & \delta \\
0 & 0 & - (\mu + \beta + \eta(1 + k_{1} )) & \psi & 0 & 0 \\
0 & 0 & \beta &  -(\mu + \gamma + \alpha(1 + k_{2} ) & 0 & 0 \\
0 & 0 & \eta(1 + k_{1} ) & \alpha(1 + k_{2} ) & -(\mu + \omega ) & 0 \\
0 & 0 & 0 & 0 & \omega & -(\mu + \delta ) 
\end{pmatrix} .\]
The Jacobian of the linearised system has a simple zero eigenvalues, 
with all other eigenvalues having negative real parts, hence the center 
manifold theory can be used to analyse the dynamics of the system 
around the bifurcation point $\psi ,$  \cite{Sha} and \cite{Van}. The 
Jacobian matrix has a right eigenvectors (corresponding to the zero 
eigenvalues)  given by
\begin{equation*} 
h = ( h_{1} , h_{2} , h_{3} , h_{4} , h_{5} , h_{6} )
\end{equation*}
\begin{eqnarray*} 
 -(\mu + \upsilon ) h_{1}  =  0  \Rightarrow h_{1}  =  0 \\
 \upsilon h_{1} - \mu h_{2} - \psi h_{4} + \delta h_{6} = 0
 \Rightarrow   h_{2}  =  \frac{\delta h_{6} - \psi h_{4}}{\mu} \\
h_{3}  =  \frac{\psi h_{4}}{\mu + \beta + \alpha(1 + k_{2} )} \\
\beta h_{3} - (\mu + \gamma + \alpha(1 + k_{2} ) h_{4} = 0
  \Rightarrow  h_{4}  =  \frac{\beta h_{3}}{\mu + \gamma + \alpha(1 + 
  k_{2} )} \\ 
h_{5}  =  \frac{\eta(1 + k_{1} )h_{3}  + \alpha(1 + k_{2} )h_{4}}{\mu + 
\omega} \\
\omega h_{5} -(\mu + \delta) h_{6} 0 
 \Rightarrow  h_{6}  =  \frac{\omega}{\mu + \delta} .
\end{eqnarray*}
Similarly, the left eigenvectors (corresponding to the zero 
eigenvalues) are given by
 \begin{equation*}
v = ( v_{1} , v_{2} , v_{3} , v_{4} , v_{5} , v_{6} )
\end{equation*}
where 
\begin{eqnarray*} 
-\mu v_{2} = 0
 \Rightarrow  v_{2} & =& 0 \\ 
-(\mu + \upsilon )v_{1} + \upsilon v_{2} = 0 
 \Rightarrow v_{1} & = & 0 \\
\delta v_{2} - (\mu + \delta ) v_{ 6} = 0 
\Rightarrow  v_{ 6} & =& 0 \\ 
-(\mu + \omega ) v_{5} = 0
 \Rightarrow v_{5} & =& 0 \\
-(\mu + \beta + \alpha(1 + k_{2} )v_{3} + \beta v_{4} + 
\eta(1 + k_{1} ) v_{5} & =& 0 \\
 \Rightarrow v_{3} & =& \frac{\beta v_{4}}{\mu + \beta + \alpha (1 + 
 k_{2} )} \\
-\psi v_{2} + \psi v_{3} - (\mu + \gamma + \alpha(1 + k_{2} )) v_{4} + 
\alpha(1 + k_{2} )) v_{5}  =  0 
 \Rightarrow v_{4} & =& \frac{\psi v_{3}}{\mu + \gamma + \alpha(1 + 
 k_{2} )} \qquad .
\end{eqnarray*}
So that $v\cdot h = 1$ in line with \cite{Gar}.

We are now left to  consider $f_{k}; ~~ k = 3 , 4$  since $v_{1} = 
v_{2} = v_{5} =v_{6} = 0 .$
 
 The local dynamic of the system is totally govern by the signs of $a$ 
 and $b$ . For instance, if $a = 0$, and $b > 0$   when $\psi < 0$, 
 then, $0$ is locally asymptotically stable and there exist a positive 
 stable equilibrium \cite{Sac}.
Hence, by computing the non-zero partial derivatives of the right-hand 
function $f_{i},$  $i = 1, 2, \cdots, 6,$  the associated backward 
bifurcation coefficients $a$ and $b$ are given respectively by 
\[ a =  \sum_{i = j = k = 1}^n  v_{k} h_{i} h_{j}
 \frac{\partial^2 f_{k}}
 {\partial x_{i} x_{j}} (0,0).\]
So,   
 \begin{eqnarray*}
\frac{\partial^2 f_{3}}{\partial x_{3} \partial x_{4}}  = 0 ~~ 
\text{and} ~~
\frac{\partial^2 f_{4}}{\partial x_{3} \partial x_{4}} = 0 .
\end{eqnarray*}
This implies  
 \begin{eqnarray*}
v_{3} h_{4}\frac{\partial^2 f_{3}}{\partial x_{3} \partial x_{4}}   +
 v_{4} h_{3}\frac{\partial^2 f_{4}}{\partial x_{3} \partial x_{4}} = 0 
  \Rightarrow a = 0 .
\end{eqnarray*}
 and
 \begin{equation*}
 b =  \sum_{i = j = k = 1}^n  v_{k} h_{i}  
 \frac{\partial^2 f_{k}}{\partial x_{i} x_{\psi}} (0,0) 
 \end{equation*}
 with 
\begin{eqnarray*}
\frac{\partial^2 f_{3}}{\partial x_{3} \partial \psi}  = 1 ~~ \text{and} ~~
\frac{\partial^2 f_{4}}{\partial x_{4} \partial \psi} = 0 .
\end{eqnarray*}
So,  
\begin{eqnarray*}
   v_{3} h_{3}\frac{\partial^2 f_{3}}{\partial x_{3} \partial \psi}   +
  v_{4} h_{4}\frac{\partial^2 f_{4}}{\partial x_{4} \partial \psi} 
  = 1 + 0 = 1 > 0 \Rightarrow b > 0 .
\end{eqnarray*}
 
Since the backward bifurcation co-efficient $b$ is positive,   it 
follows that the  gonorrhoea model will undergo backward bifurcation. 
 This means that there is Endemic Equilibrium  when $R_{e} > 1$,  and  
 when $R_{e} = 1.$   But from equations of  $R_{0}$ and $R_{e}$, they 
 are both less than $1,$ showing that the disease will be controlled  
 in the population in a limited time.
\end{proof}

\section{Discussion of Results}
Graphical simulation buttress our results. These are the following: 
\begin{figure}[hbtp]
\centering
\includegraphics[scale=0.6]{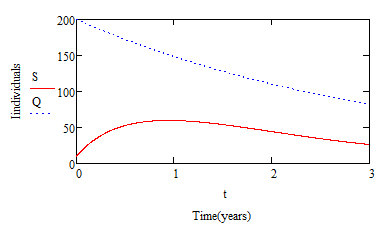}
\caption{Effect of decreasing waning rate on the susceptible and immune 
classes, i.e., $\upsilon = 0.2$. } 
\label{image1}
\centering
\includegraphics[scale=0.6]{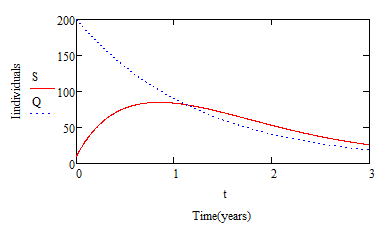}
\caption {Effect of increasing waning rate on the susceptible and 
immune classes, i.e.,  $\upsilon = 0.6$.}
\label{image2}
\end{figure}

Figure \ref{image1} suggests that when the waning rate $\upsilon$ is 
low (i.e., $\upsilon = 0.2$), the passive immune population decreases 
exponentially with time, while Figure \ref{image2} indicates that as 
the waning rate is high, (i.e., $\upsilon = 0.6$), the passive immune 
population decreases faster  and varnishes with time. The continuous 
decay in the population of the immune class (Q) with time is due to the 
fact that the immunity conferred on the individuals in this class is 
temporal and hence, expires with time.

However, the susceptible population increases  slower to the turning 
point at about one year and three months as the waning rate $\upsilon$ 
is low and increases faster  as the waning rate $\upsilon$ is high as 
shown in Figures  \ref{image1} and \ref{image2} respectively. In 
both cases, the susceptible class later decreases with time due to the 
interaction among the latent, infected and the susceptible classes 
coupled with the natural mortality rate $\mu .$

The impact of contact rate on Susceptible, Latent and Infected classes 
is shown in Figure \ref{image5}:
\begin{figure}[hbtp]
\centering
\includegraphics[scale=0.8]{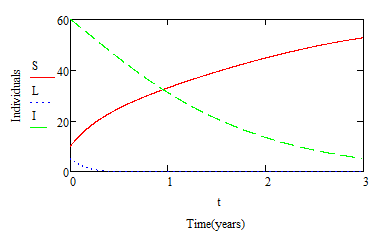}
\caption{The Effect of reducing contact rate $\lambda = \theta I$}
\label{image5}
\end{figure} 
Figure \ref{image5} indicates that when the interaction rate 
is low (i.e., $\theta =0.3$), the latent and the infected classes 
decrease exponentially with time, and even varnishes in the long run 
since there will be almost nobody to contact  and suffer 
the disease. It is also shown that when the interaction rate 
$\theta = 0$, the reproduction number of the disease becomes zero.  
That is, 
\[R_{0} = \frac{(\beta \theta S)}{(\beta + \mu + \eta)(\mu + 
\gamma + \alpha)} = 0.\] 
Thus, at this point, the contact rate $\lambda$ becomes zero and hence, 
nobody suffers the disease.

\section{Conclusion}
Based on the analysis and results of this work, we observed 
that the disease would be eradicated from the population since the 
effective reproduction number is less than $1.$ Again, addition of 
treatment or control measures such as condom and education 
enlightenment helped to reduce the infection in the population.  
However, addition of control parameters led to transcritical 
bifurcation.

From the graphical illustrations, we concluded that immune population 
continues to decay exponentially due to temporal immunity conferred on 
the individuals in the immune class. We also concluded that 
reproduction number of the infection grows when there is no control 
measure in the model and decays when control measure is applied in the 
model. Finally, we concluded that for the disease to be totally 
eliminated from the community,  the interaction rate $\theta$ with the 
infective which leads to contacts should be totally reduced to the 
barest minimum or zero.

\end{document}